\documentclass[a4paper,10pt]{amsart}
\usepackage{amsmath,amsthm,amssymb,latexsym,enumerate,color,hyperref}
\usepackage{graphicx}

\numberwithin{equation}{section}

\begin{document}

\newtheorem{thm}{Theorem}[section]
\newtheorem{prop}[thm]{Proposition}
\newtheorem{lem}[thm]{Lemma}
\newtheorem{cor}[thm]{Corollary}
\newtheorem{rem}[thm]{Remark}
\newtheorem*{defn}{Definition}

\newcommand{\DD}{\mathbb{D}}
\newcommand{\NN}{\mathbb{N}}
\newcommand{\ZZ}{\mathbb{Z}}
\newcommand{\QQ}{\mathbb{Q}}
\newcommand{\RR}{\mathbb{R}}
\newcommand{\CC}{\mathbb{C}}
\renewcommand{\SS}{\mathbb{S}}

\renewcommand{\theequation}{\arabic{section}.\arabic{equation}}

\newcommand{\supp}{\mathop{\mathrm{supp}}}    

\newcommand{\re}{\mathop{\mathrm{Re}}}   
\newcommand{\im}{\mathop{\mathrm{Im}}}   
\newcommand{\dist}{\mathop{\mathrm{dist}}}  
\newcommand{\link}{\mathop{\circ\kern-.35em -}}
\newcommand{\spn}{\mathop{\mathrm{span}}}   
\newcommand{\ind}{\mathop{\mathrm{ind}}}   
\newcommand{\rank}{\mathop{\mathrm{rank}}}   
\newcommand{\Fix}{\mathop{\mathrm{Fix}}}   
\newcommand{\codim}{\mathop{\mathrm{codim}}}   
\newcommand{\conv}{\mathop{\mathrm{conv}}}   
\newcommand{\epsi}{\mbox{$\varepsilon$}}
\newcommand{\eps}{\mathchoice{\epsi}{\epsi}
{\mbox{\scriptsize\epsi}}{\mbox{\tiny\epsi}}}
\newcommand{\cl}{\overline}
\newcommand{\pa}{\partial}
\newcommand{\ve}{\varepsilon}
\newcommand{\zi}{\zeta}
\newcommand{\Si}{\Sigma}
\newcommand{\cA}{{\mathcal A}}
\newcommand{\cG}{{\mathcal G}}
\newcommand{\cH}{{\mathcal H}}
\newcommand{\cI}{{\mathcal I}}
\newcommand{\cJ}{{\mathcal J}}
\newcommand{\cK}{{\mathcal K}}
\newcommand{\cL}{{\mathcal L}}
\newcommand{\cN}{{\mathcal N}}
\newcommand{\cR}{{\mathcal R}}
\newcommand{\cS}{{\mathcal S}}
\newcommand{\cT}{{\mathcal T}}
\newcommand{\cU}{{\mathcal U}}
\newcommand{\OM}{\Omega}
\newcommand{\B}{\bullet}
\newcommand{\ol}{\overline}
\newcommand{\ul}{\underline}
\newcommand{\vp}{\varphi}
\newcommand{\AC}{\mathop{\mathrm{AC}}}   
\newcommand{\Lip}{\mathop{\mathrm{Lip}}}   
\newcommand{\es}{\mathop{\mathrm{esssup}}}   
\newcommand{\les}{\mathop{\mathrm{les}}}   
\newcommand{\nid}{\noindent}
\newcommand{\pzr}{\phi^0_R}
\newcommand{\pir}{\phi^\infty_R}
\newcommand{\psr}{\phi^*_R}
\newcommand{\pow}{\frac{N}{N-1}}
\newcommand{\ncl}{\mathop{\mathrm{nc-lim}}}   
\newcommand{\nvl}{\mathop{\mathrm{nv-lim}}}  
\newcommand{\la}{\lambda}
\newcommand{\La}{\Lambda}    
\newcommand{\de}{\delta}    
\newcommand{\fhi}{\varphi} 
\newcommand{\ga}{\gamma}    
\newcommand{\ka}{\kappa}   

\newcommand{\core}{\heartsuit}
\newcommand{\diam}{\mathrm{diam}}

\newcommand{\lan}{\langle}
\newcommand{\ran}{\rangle}
\newcommand{\tr}{\mathop{\mathrm{tr}}}
\newcommand{\diag}{\mathop{\mathrm{diag}}}
\newcommand{\dv}{\mathop{\mathrm{div}}}

\newcommand{\al}{\alpha}
\newcommand{\be}{\beta}
\newcommand{\Om}{\Omega}
\newcommand{\na}{\nabla}

\newcommand{\cC}{\mathcal{C}}
\newcommand{\cM}{\mathcal{M}}
\newcommand{\nr}{\Vert}
\newcommand{\De}{\Delta}
\newcommand{\cX}{\mathcal{X}}
\newcommand{\cP}{\mathcal{P}}
\newcommand{\om}{\omega}
\newcommand{\si}{\sigma}
\newcommand{\te}{\theta}
\newcommand{\Ga}{\Gamma}

\title[Serrin's problem and Alexandrov's SBT]{Serrin's problem and \\ Alexandrov's Soap Bubble Theorem: \\ enhanced stability via integral identities}

\author{Rolando Magnanini} 
\address{Dipartimento di Matematica ed Informatica ``U.~Dini'',
Universit\` a di Firenze, viale Morgagni 67/A, 50134 Firenze, Italy.}
    \email{magnanin@math.unifi.it}
    \urladdr{http://web.math.unifi.it/users/magnanin}

\author{Giorgio Poggesi}
\address{Dipartimento di Matematica ed Informatica ``U.~Dini'',
Universit\` a di Firenze, viale Morgagni 67/A, 50134 Firenze, Italy.}
    \email{giorgio.poggesi@unifi.it}


\begin{abstract}
We consider Serrin's overdetermined problem for the torsional rigidity and Alexandrov's Soap Bubble Theorem. We present new integral identities, that show a strong analogy between the two problems and help to obtain better (in some cases optimal) quantitative estimates for the radially symmetric configuration. The estimates for the Soap Bubble Theorem benefit from those of Serrin's problem.
\end{abstract}

\keywords{Serrin's overdetermined problem, Alexandrov Soap Bubble Theorem, torsional rigidity, constant mean curvature, integral identities, stability, quantitative estimates}
\subjclass{Primary 35N25, 53A10, 35B35; Secondary 35A23}

\maketitle

\raggedbottom

\section{Introduction}
The pioneering symmetry results obtained by A. D. Alexandrov \cite{Al1}, \cite{Al2} and  J. Serrin \cite{Se} are now classical but still influential. The former --- the well-known {\it Soap Bubble Theorem} --- states that a compact hypersurface, embedded in $\RR^N$, that has constant mean curvature must be a sphere. The latter --- Serrin's symmetry result --- has to do with certain overdetermined problems for partial differential equations.
In its simplest formulation, it states that  the overdetermined boundary value problem 
\begin{eqnarray}
\label{serrin1}
&\De u=N \ \mbox{ in } \ \Om, \quad u=0 \ \mbox{ on } \ \Ga, \\
\label{serrin2}
&u_\nu=R \ \mbox{ on } \ \Ga, 
\end{eqnarray}
admits a solution 
for some positive constant $R$ if and only if $\Om$ is a ball of radius $R$ and, up to translations, $u(x)=(|x|^2-R^2)/2$. Here, $\Om$ denotes a bounded domain in $\RR^N$, $N\ge 2$, with sufficiently smooth boundary $\Ga$, say $C^2$, and $u_\nu$ is the outward normal derivative of $u$ on $\Ga$. 
This result inaugurated a new and fruitful field in mathematical research at the confluence of Analysis and Geometry and has many applications to other areas of mathematics and natural sciences. To be sure, that same result was actually motivated by two concrete problems in mathematical physics regarding the torsion of a straight solid bar and the tangential stress of a fluid on the walls of a rectilinear pipe.
\par
The two problems share several common features. To prove his result, Alexandrov introduced his {\it reflection principle}, an elegant geometric technique that also works for other symmetry results concerning curvatures. 
Serrin's proof hinges on his  {\it method of moving planes} --- an adaptation and refinement of the reflection principle --- that proves to be a very flexible tool, since it allows to prove
radial symmetry for {\it positive} solutions of a far more general class of non-linear equations, that includes the semi-linear equation
\begin{equation}
\label{semilinear}
\De u=f(u),
\end{equation}
where $f$ is a locally Lipschitz continuous non-linearity. 
\par
Also, alternative proofs of both symmetry results can be given, based on certain integral identities and inequalities, and the maximum principle. 
H.~F.~Weinberger's proof (\cite{We}) of symmetry,
even if it is known to work only for problem \eqref{serrin1}-\eqref{serrin2} and other few instances, leaves open the option of considering less regular settings. Such a possibility may be extended to the Soap Bubble Theorem, by using R.~C.~Reilly's argument (\cite{Re}), in which the relevant hypersurface is regarded as a level surface of the solution of \eqref{serrin1} ($\Ga$ itself, indeed).
\par
In this paper, we shall further analyze the analogies in Weinberger's and Reilly's arguments to obtain quantitative estimates of the desired radial symmetries. Roughly speaking, we shall address the problem of estimating how close to a sphere is a hypersurface $\Ga$, if either its mean curvature $H$ is close to be constant or, alternatively, if the normal derivative on $\Ga$ of the solution of \eqref{serrin1} is close to be constant.
\par 
In both problems, the radial symmetry of $\Ga$ will descend from that of the solution $u$ of \eqref{serrin1}, by showing that in {\it Newton's inequality} 
\begin{equation}
\label{newton}
(\De u)^2\le N\,|\na^2 u|^2,
\end{equation}
that, by Cauchy-Schwarz inequality, holds pointwise in $\Om$, the equality sign is identically attained in $\Om$. In fact, such an occurrence holds if and only if 
$u$ is a quadratic polynomial $q$ of the form
\begin{equation}
\label{quadratic}
q(x)=\frac12\, (|x-z|^2-a),
\end{equation}
for some choice of $z\in\RR^N$ and $a\in\RR$. The boundary condition in \eqref{serrin1} will then
tells us that $\Ga$ must be a sphere centered at $z$.
\par
The starting points of our analysis are the following two integral identities:
\begin{equation}
\label{idwps}
\int_{\Om} (-u) \left\{ |\na ^2 u|^2- \frac{ (\De u)^2}{N} \right\} dx=
\frac{1}{2}\,\int_\Ga \left( u_\nu^2- R^2\right) (u_\nu-q_\nu)\,dS_x
\end{equation}
and
\begin{multline}
\label{identity-SBT2}
\frac1{N-1}\int_{\Om} \left\{ |\na ^2 u|^2-\frac{(\De u)^2}{N}\right\}dx+
\frac1{R}\,\int_\Ga (u_\nu-R)^2 dS_x = \\
\int_{\Ga}(H_0-H)\,(u_\nu-q_\nu)\,u_\nu\,dS_x+
\int_{\Ga}(H_0-H)\, (u_\nu-R)\,q_\nu\, dS_x.
\end{multline}
\par
The two identities hold regardless of how the point $z$ or the constant $a$ are chosen in \eqref{quadratic}. In \eqref{idwps} and \eqref{identity-SBT2}, $R$ and $H_0$ are reference constants given by
\begin{equation}
\label{def-R-H0}
R=\frac{N\,|\Om|}{|\Ga|}, \quad H_0=\frac1{R}=\frac{|\Ga|}{N\,|\Om|}.
\end{equation}
\par
If $u_\nu$ is constant on $\Ga$, that constant must be equal to $R$, by the identity
\begin{equation}
\label{divergence}
\int_\Ga u_\nu\,dS_x=N\,|\Om|,
\end{equation}
and hence we obtain symmetry by using \eqref{idwps}, since the equality sign must hold in \eqref{newton} (in fact, $-u>0$ in $\Om$ by the maximum principle).
If in turn $H$ is constant on $\Ga$, then the {\it Minkowski's identity},
\begin{equation}
\label{minkowski}
\int_\Ga H\,q_\nu\,dS_x=|\Ga|,
\end{equation}
implies that $H\equiv H_0$ on $\Ga$, and hence once again symmetry follows by an inspection of \eqref{identity-SBT2}, instead.
\par
Identity \eqref{idwps}, which appears without proof in \cite[Remark 2.5]{MP}, puts together Weinberger's identities and some remarks of L.~E.~Payne and P.~W.~Schaefer \cite{PS}; \eqref{identity-SBT2} is a slight modification of one that was proved in \cite[Theorem 2.2]{MP} and will turn out to be useful to improve our desired quantitative estimates. For the reader's convenience, we will present the proofs of 
\eqref{idwps} and \eqref{identity-SBT2} in Section \ref{sec:identity}.
\par 
Identity \eqref{idwps} certainly holds under Serrin's smoothness assumptions ($\Ga\in C^2$), but it is clear that it can be easily extended by approximation to the case of $\Ga\in C^{1,\al}$, since in that case $u_\nu$ is continuous on $\Ga$. Nevertheless, it should be noticed that, to infer the radial symmetry of $|\Om|$, it suffices to show that the surface integral on the right-hand side of \eqref{idwps} is zero. 
\par
The main motivation of this paper is to investigate how the use of \eqref{idwps} and \eqref{identity-SBT2} benefits the study of the stability of the radial configuration in the Soap Bubble Theorem  and Serrin's problem. 
Technically speaking, one may look for two concentric balls $B_{\rho_i}$ and $B_{\rho_e}$, with radii $\rho_i$ and $\rho_e$, such that
\begin{equation}
\label{balls}
B_{\rho_i} \subset \Om \subset B_{\rho_e}
\end{equation} 
 and
\begin{equation}
\label{stability}
\rho_e-\rho_i\le \psi(\eta),
\end{equation}
where
$\psi:[0,\infty)\to[0,\infty)$ is a continuous function vanishing at $0$ and $\eta$ is a suitable measure of the deviation of $u_\nu$ or $H$ from being a constant.
\par
That kind of problem has been considered for Serrin's problem for the first time in \cite{ABR}. There, for a $C^{2,\al}$-regular domain $\Om$, it is considered a {\it positive} solution $u$ of \eqref{semilinear}, such that $u=0$ on $\Ga$, and it is proved that \eqref{balls} and \eqref{stability} hold with $\psi(\eta)=C\,|\log \eta|^{-1/N}$
and $\eta=\nr u_\nu-c\nr_{C^1(\Ga)}$. (Here and in the remaining paragraphs of this introduction, $C$ denotes a positive constant that depends on some geometric and 
regularity parameters associated to $\Om$ and, when applicable, $f$.)
The proof is based on a quantitative study of the method of moving planes. 
In \cite{CMV}, in the same general framework and with a similar proof, that stability estimate has been improved. There, it is in fact shown that \eqref{balls} and \eqref{stability} hold with $\psi(\eta)=C\,\eta^\tau$, where $\eta$ is the Lipschitz seminorm $[u_\nu]_{\Ga}$ defined by
\begin{equation*}
\label{seminorm normal derivative}
\sup_{\substack{x,y \in \Ga\\ \ x \neq y}} \frac{|u_\nu(x) - u_\nu(y)|}{|x-y|}.
\end{equation*}
The exponent $\tau\in (0,1)$ can be computed for a general setting and, if $\Om$ is convex,
is arbitrarily close to $1/(N+1)$.
\par
Our stability estimate for the original Serrin's problem \eqref{serrin1}-\eqref{serrin2} improves on a different technique, based on Weinberger's integral identities and  first employed in \cite{BNST}. There, a quantitative estimate of H\"older type was obtained for the first time, by using the weaker deviation $\eta=\nr u_\nu-c\nr_\infty$ of $u_\nu$ from some reference constant.  In \cite{BNST}, it is also considered the possibility to measure the deviation of $u_\nu$ from a constant by the $L^1$-norm and it is shown that $\Om$ can be approximated in measure by a finite number of mutually disjoint balls $B_i$. The error in the approximation is 
$\psi(\eta)=C\,\eta^\tau$, with $\eta=\nr u_\nu-c\nr_{1,\Ga}$ and $\tau=1/(4N+9)$. Recently, that approach has been greatly improved in \cite{Fe} where, rather than by \eqref{balls} and \eqref{stability}, the closeness of $\Om$ to a ball is measured by the following slight modification of the so-called Fraenkel asymmetry:
\begin{equation}
\label{asymmetry}
\cA(\Om)=\inf\left\{\frac{|\Om\De B^x|}{|B^x|}: x \mbox{ center of a ball $B^x$ with radius $R$} \right\},
\end{equation}
where $\Om\De B^x$ denotes the symmetric difference of $\Om$ and $B^x$, and $R$ is the constant defined in \eqref{def-R-H0}.
In \cite{Fe}, the deviation of $u_\nu$ from a constant is measured in $L^2$-norm and the estimate obtained is of Lipschitz type: $\cA(\Om)\le C\,\nr u_\nu-R\nr_{2,\Ga}$.
\par
In this paper, we turn back to an approximation of type \eqref{balls} and \eqref{stability} and, in Theorems \ref{thm:Improved-Serrin-stability} and \ref{thm:Serrin-stability}, we derive two inequalities:
\begin{equation}
\label{stability-impr-gen}
\rho_e-\rho_i\le C\,\nr u_\nu-R\nr_{2, \Ga}^{2/(N+2)}
\end{equation}
and
\begin{equation}
\label{stability-Serrin}
\rho_e-\rho_i\le C\,\nr u_\nu-R\nr_{1,\Ga}^\frac{1}{N+2}.
\end{equation}
The latter clearly improves that obtained in \cite{BNST}, while the former makes better than \cite{CMV}, even if we replace the Lipschitz semi-norm by an $L^2$-deviation.
\par
Another salient result in our paper is a quantitative inequality for simmetry in the Soap Bubble Theorem, that is obtained as a benefit from the analysis employed to derive \eqref{stability-impr-gen}.  In fact, in Theorem \ref{thm:SBT-improved-stability} , we prove that 
\begin{equation}
\label{SBT-the-estimate}
\rho_e-\rho_i\le C\,\nr H_0-H \nr_{2,\Ga}^{\tau_N},
\end{equation}
where $\tau_N=1$ for $N=2, 3$ and $\tau_N=2/(N+2)$ for $N\ge 4$. 
Inequality \eqref{SBT-the-estimate} enhances, in all dimensions, estimates obtained in \cite{CM} with the uniform deviation $\nr H_0-H\nr_{\infty,\Ga}$, for strictly mean convex surfaces. 
Also, it improves by a factor of $2$ the results derived by the authors in \cite{MP}.
More importantly, for the cases $N=2, 3$ but with a weaker deviation and for a more general class of hypersurfaces, gains the (optimal) Lipschitz stability proved in \cite{CV} and \cite[Theorem 1.8]{KM} for strictly mean convex hypersurfaces, by using the uniform deviation. 
\par
As a final important achievement, by arguments similar to those of \cite{Fe}, we also get in Theorem \ref{th:asymmetry} the (optimal) inequality for the asymmetry \eqref{asymmetry}:
$$
\cA(\Om)\le C\,\nr H_0-H\nr_{2,\Ga}.
$$

\medskip

In the remaining paragraphs of this introduction, we shall pinpoint the salient remarks that lead us to the proof of \eqref{stability-impr-gen} and \eqref{SBT-the-estimate}.
\par
We start to simplify matters 
by noticing that, by \eqref{idwps}, the {\it harmonic} function $h=q-u$ satisfies:
\begin{equation}
\label{idwps-h}
\int_{\Om} (-u)\, |\na ^2 h|^2\,dx=
\frac{1}{2}\,\int_\Ga ( R^2-u_\nu^2)\, h_\nu\,dS_x.
\end{equation}
Also, notice that $h=q$ on $\Ga$ and hence
\begin{equation}
\label{oscillation}
\max_{\Ga} h-\min_{\Ga} h=\frac12\,(\rho_e^2-\rho_i^2)\ge \frac12\,(|\Om|/|B|)^{1/N}(\rho_e-\rho_i).
\end{equation}
\par
Now, observe that \eqref{idwps-h} holds regardless of the choice of the parameters $z$ and $a$ defining $q$, so that we will complete the first step of our proof by choosing $z\in\Om$ in a way that the {\it oscillation} of $h$ on $\Ga$ on the left-hand side of \eqref{oscillation} (which is indeed the oscillation of $h$ on $\ol{\Om}$)
can be bounded by a suitable function of the left-hand side of \eqref{idwps-h}.
\par
\par
To carry out this plan, we use three ingredients. First, we choose $z\in\Om$ as a minimum (or any critical) point of $u$ and, as done in \cite[Lemma 3.3]{MP}, and we show that 
$$
\max_{\Ga} h-\min_{\Ga} h \le C\,\left(\int_\Om h^2 dx\right)^{1/(N+2)}.
$$
Secondly, we observe that, depending on the regularity of $\Ga$, we can easily obtain the bound
$$
C\,\de_\Ga(x)^\al\le -u \ \mbox { on } \ \ol{\Om},
$$
where $\de_\Ga(x)=\dist(x,\Ga)$ for $x\in\ol{\Om}$ and $\al=1$ or $2$.
Thirdly, we apply to the first (harmonic) derivatives of $h$ the Hardy-Poincar\'e-type inequality 
\begin{equation*}
\label{boas-straube}
\int_\Om v(x)^2 dx\le C \int_\Om \de_\Ga(x)^\al |\na v(x)|^2 dx,
\end{equation*}
that holds for fixed $\al\in[0,1]$ and for any harmonic function $v\in W^{1,2}(\Om)$ that is zero at some given point in $\Om$ (in our case that point will be $z$)
and, to $h$, the Poincar\'e-type inequality
\begin{equation*}
\label{classicalpoincarè}
\int_\Om v^2 dx\le C \int_\Om |\na v|^2 dx,
\end{equation*}
that holds for any harmonic function $v\in W^{1,2}(\Om)$ with zero mean value on $\Om$. 
Thus, summing up the last $5$ inequalities gives:
\begin{equation}
\label{the-estimate-for-h}
\max_{\Ga} h-\min_{\Ga} h\le C \left(\int_{\Om} (-u)\, |\na ^2 h|^2\,dx\right)^{1/(N+2)}.
\end{equation}
\par
Next, we work on the right-hand side of \eqref{idwps-h}. An important observation is that, if $u_\nu-R$ tends to $0$, also $h_\nu$ does. Quantitavely, this fact can be expressed by the inequality
$$
\nr h_\nu\nr_{2,\Ga}\le C\,\nr u_\nu-R\nr_{2,\Ga},
$$
that can be derived from \cite{Fe}. 
Thus, \eqref{stability-impr-gen} will follow by using this inequality, after an application of H\"older's inequality to the right-hand side of \eqref{idwps-h}.
\par
In order to prove \eqref{SBT-the-estimate}, we use the new identity \eqref{identity-SBT2}. In fact, discarding the first summand at its left-hand side and applying H\"older's and the last inequality to its right-hand side yield that
$$
\nr u_\nu-R\nr_{2,\Ga}\le C\,\nr H_0-H\nr_{2,\Ga}.
$$
Inequality \eqref{SBT-the-estimate} then follows again from \eqref{identity-SBT2} and the estimate 
$$
\max_{\Ga} h-\min_{\Ga} h\le C \left(\int_{\Om} |\na ^2 h|^2\,dx\right)^{\tau_N/2}
$$
already obtained in \cite{MP}.
\par
The constants $C$ in our inequalities generally depend on $N$, the {\it volume} $|\Om|$ and the {\it diameter} $d_\Om$ of $\Om$, and the radii $r_i$ and $r_e$ of the largest balls osculating $\Ga$ and, respectively, contained in $\Om$ or in $\RR^N\setminus\ol{\Om}$. Their dependence on those parameters will be clarified later.
 
\par
The paper is organized as follows.
In Section \ref{sec:identity}, we collect the relevant identities on which our results are based. 
Section \ref{sec:stability} is dedicated to the stability of Serrin's problem. Subsection \ref{subsec:harmonicestimates} contains the estimates on harmonic functions that are instrumental to derive \eqref{stability-Serrin} and \eqref{stability-impr-gen}. Then, in
Subsection \ref{subsec:improvedstability}, we assemble all the relevant identities and inequalities to estabilish \eqref{stability-Serrin} and \eqref{stability-impr-gen} (respectively Theorems \ref{thm:Improved-Serrin-stability} and \ref{thm:Serrin-stability}).
In Section \ref{sec:SBT-improvedstab}, we present the new stability result \eqref{SBT-the-estimate} for Alexandrov's Soap Bubble Theorem (Theorem \ref{thm:SBT-improved-stability}).

\section{Identities for Serrin's problem and the Soap Bubble Theorem}
\label{sec:identity}

We begin by setting some relevant notations. By $\Om\subset\RR^N$, $N\ge 2$, we shall denote a bounded domain and call $\Ga$ its boundary. 
By $|\Om|$ and $|\Ga|$, we will denote indifferently the $N$-dimensional Lebesgue measure of $\Om$
and the surface measure of $\Ga$. When $\Ga$ is of class $C^1$, $\nu$ will denote the (exterior) unit normal vector field to $\Ga$ and, when $\Ga$ is a hypersurface of class $C^2$, $H(x)$ will denote its mean curvature  (with respect to $-\nu(x)$) at $x\in\Ga$. 
\par
We will also use the letter $q$ to denote the quadratic polynomial defined in \eqref{quadratic}, where $z$ is any point in $\RR^N$ and $a$ is any real number; furthermore, we will always use the letter $h$ to denote the harmonic function
$$
h=q-u.
$$

Finally, as already mentioned in the introduction, for a point $z\in\Om$ to be determined, $\rho_i$ and $\rho_e$ shall denote 
the radius of the largest ball contained in $\Om$
and that of the smallest ball that contains $\Om$, both centered at $z$; in formulas, 
\begin{equation}
\label{def-rhos}
\rho_i=\min_{x\in\Ga}|x-z| \ \mbox{ and } \ \rho_e=\max_{x\in\Ga}|x-z|.
\end{equation}
\par
To start, we will provide the proof of identity \eqref{idwps} appearing in \cite[Remark 2.5]{MP}. 
The proof is as minimal as possible. It follows the tracks of and improves on the work of Weinberger \cite{We} and its modification due to Payne and Schaefer \cite{PS}. 
\par
Identity \eqref{idwps} is a consequence of two ingredients: the differential identity for the solution $u$ of \eqref{serrin1},
\begin{equation}
\label{differential-identity}
|\na^2 u|^2-\frac{(\De u)^2}{N}=\De P ,
\end{equation}
that associates the Cauchy-Schwarz deficit on the left-hand side with the {\it P-function},
\begin{equation*}
\label{P-function}
P = \frac{1}{2}\,|\nabla u|^2 - u,
\end{equation*}
and is easily obtained by direct computation; the {\it Rellich-Pohozaev} identity for the solution $u$ of \eqref{serrin1} (see \cite{Po}):
\begin{equation}
\label{Pohozaev}
(N+2) \int_{\Om} |\na u|^2 \, dx =\int_\Ga (u_\nu)^2 \, q_\nu\,dS_x.
\end{equation}
Notice that \eqref{differential-identity} also implies
that $P$ is subharmonic, since the left-hand side is non-negative by Cauchy-Schwarz inequality.

\begin{thm}[Fundamental identity for Serrin's problem] 
\label{th:wps}
Let $\Om \subset \mathbb R^N$ be a bounded domain with boundary $\Ga$ of class $C^{1,\al}$, $0<\al\le 1$, and $R$ be the positive constant defined in \eqref{def-R-H0}.
\par
Then the solution $u$ of  \eqref{serrin1} satisfies identity \eqref{idwps}:
\begin{equation*}
\int_{\Om} (-u)\,\left\{ |\na ^2 u|^2- \frac{ (\De u)^2}{N} \right\}\,dx=
\frac{1}{2}\,\int_\Ga \left( u_\nu^2 - R^2 \right) \,(u_\nu-q_\nu)\,dS_x.
\end{equation*}
\par
Therefore, if the right-hand side of \eqref{idwps} is non-positive,
$\Ga$ must be a sphere (and hence $\Om$ a ball) of radius $R$. 
The same conclusion holds, in particular, if $u_\nu$ is constant on $\Ga$.

\end{thm}

\begin{proof}
First, suppose that $\Ga$ is of class $C^{2,\al}$, so that $u\in C^{2,\al}(\ol{\Om})$. Integration by parts then gives:
$$
\int_\Om (u\,\De P-P\,\De u)\,dx=\int_\Ga(u\,P_\nu-u_\nu\,P)\,dS_x;
$$
thus, since $u$ satisfies \eqref{serrin1}, we have that
\begin{equation}
\label{parts}
\int_\Om (-u)\,\De P\,dx=-N\,\int_\Om P\,dx+\frac12\,\int_\Ga u_\nu^3\,dS_x,
\end{equation}
being $P=|\na u|^2/2=u_\nu^2/2$ on $\Ga$.

Now, we compute the first summand on the right-hand side of \eqref{parts} by the divergence theorem:
$$
N \int_{\Om} P\,dx=\frac{N}{2} \int_\Om |\na u|^2 dx-\int_\Om u\,\De u\,dx =
\left( \frac{N}{2} + 1 \right) \int_{\Om} |\na u|^2 \, dx.
$$
Thus, \eqref{parts}, \eqref{differential-identity}, and \eqref{Pohozaev} give \eqref{idwps}, since in \eqref{parts} we can replace $(u_\nu)^3$ by $(u_\nu^2-R^2)\,(u_\nu-q_\nu)$, being
$$
\int_\Ga (u_\nu-q_\nu)\,dS_x=0,
$$
since $u-q$ harmonic in $\Om$.
\par
If $\Ga$ is of class $C^{1,\al}$, then $u\in C^{1,\al}(\ol{\Om})\cap C^2(\Om)$. Thus, by a standard approximation argument, we conclude that \eqref{idwps} holds also in this case.
\par
Now, if the right-hand side of \eqref{idwps} is non-positive, being the left-hand side non-negative by \eqref{newton} and the maximum principle for $u$,  then \eqref{newton} must hold with the equality sign, since $u<0$ on $\Om$, by the strong maximum principle.
Being $\De u=N$, we infer that $\na^2 u$ coincides with the identity matrix $I$.  Thus, $u$ must be a quadratic polynomial $q$ of the form \eqref{quadratic},
for some $z\in\RR^N$ and $a\in\RR$.
\par
Since $u=0$ on $\Ga$, then $|x-z|^2=a$ for $x\in\Ga$, that is $a$ must be positive and
$$
\sqrt{a}\,|\Ga|=\int_\Ga |x-z|\,dS_x=\int_\Ga (x-z)\cdot\nu(x)\,dS_x=N\,|\Om|.
$$
In conclusion, $\Ga$ must be a sphere centered at $z$ with radius $R$.
\par
Finally, if $u_\nu\equiv c$ on $\Ga$ for some constant $c$, then 
$$
c\,|\Ga|=\int_\Ga u_\nu\,dS_x=N\,|\Om|,
$$
that is $c=R$, and hence we can apply the previous argument.
\end{proof}

\begin{cor}
Let $u$ be the solution of \eqref{serrin1} and set $h=q-u$.
Then, $h$ is harmonic in $\Om$ and \eqref{idwps-h} holds true:
\begin{equation*}
\int_{\Om} (-u)\, |\na ^2 h|^2\,dx=
\frac{1}{2}\,\int_\Ga ( R^2-u_\nu^2)\, h_\nu\,dS_x.
\end{equation*}
Moreover, if the point $z$ in \eqref{quadratic} is chosen in $\Om$, then \eqref{oscillation} holds.\end{cor}
\begin{proof}
Simple computations give that $|\na^2 h|^2=|\na^2 u|^2-(\De u)^2/N$ and
$h_\nu=q_\nu-u_\nu$,
and hence \eqref{idwps-h} easily follows from \eqref{idwps}.
\par
Notice that $h=q$ on $\Ga$. Thus, the equality in \eqref{oscillation} follows from \eqref{def-rhos}, by choosing $z$ in $\Om$. The inequality in \eqref{oscillation} is implied by $\rho_e+\rho_i\ge\rho_e\ge (|\Om|/|B|)^{1/N}$, since $B_{\rho_e}\supseteq\Om$.
\end{proof}
\begin{rem}
{\rm
The assumptions on the regularity of $\Ga$ can further be weakened. For instance, if $\Om$ is a 
(bounded) convex domain, then inequality \eqref{bound-M-convex} below, imply that $u_\nu$ is essentially bounded on $\Ga$ with respect to the $(N-1)$-dimensional Hausdorff measure on $\Ga$. Thus, an approximation argument again gives that \eqref{idwps} still holds true.
}
\end{rem}
\par
For the sake of completeness, we present the proof of \eqref{identity-SBT2}, that is a modification of \cite[formula (2.6)]{MP}.

\begin{thm}[Fundamental identity for the Soap Bubble Theorem]
Let $\Om$ be a bounded domain with boundary $\Ga$ of class $C^2$. Then identity \eqref{identity-SBT2} holds true:
\begin{multline*}
\frac1{N-1}\int_{\Om} \left\{ |\na ^2 u|^2-\frac{(\De u)^2}{N}\right\}dx+
\frac1{R}\,\int_\Ga (u_\nu-R)^2 dS_x = \\
\int_{\Ga}(H_0-H)\,(u_\nu-q_\nu)\,u_\nu\,dS_x+
\int_{\Ga}(H_0-H)\, (u_\nu-R)\,q_\nu\, dS_x.
\end{multline*}
\end{thm}

\begin{proof}
We proceed similarly to the proof of Theorem \ref{th:wps}, but with two main differences: in place of Pohozaev's identity, we use Minkowski's identity; we use the following well-known formula for the laplacian of $u$:
$$
\De u=u_{\nu\nu}+(N-1) H\,u_\nu.
$$
This holds pointwise on any regular level surface of $u$, if we agree to still denote by $\nu$ the vector field $\na u/|\na u|$; it is clear that, on $\Ga$, this coincides with the normal.   
\par
We begin with the divergence theorem:
$$
\int_\Om \De P\,dx=\int_\Ga P_\nu\,dS_x.
$$
To compute $P_\nu$, we observe that $\na u$ is
parallel to $\nu$ on $\Ga$, that is $\na u=(u_\nu)\,\nu$ on $\Ga$. Thus,
\begin{multline*}
P_\nu=\lan \na^2 u\, \na u, \nu\ran-u_\nu=u_\nu \lan (\na^2 u)\,\nu,\nu\ran-u_\nu=u_{\nu\nu}\, u_\nu-u_\nu=\\
u_\nu [\De u-(N-1) H u_\nu]-u_\nu=(N-1)\,(1-H u_\nu)\,u_\nu,
\end{multline*}
where we have used \eqref{serrin1}. Therefore,
\begin{equation}
\label{intermediate}
\frac1{N-1}\,\int_\Om \De P\,dx=\int_\Ga (1-H u_\nu)\,u_\nu\,dS_x.
\end{equation}
Now, straightforward calculations that use \eqref{def-R-H0}, \eqref{minkowski}, and \eqref{divergence}  tell us that 
\begin{multline*}
\int_{\Ga}(H_0-H)\,(u_\nu-q_\nu)\,u_\nu\,dS_x+\\
\int_{\Ga}(H_0-H)\, (u_\nu-R)\,q_\nu\, dS_x=
H_0 \int_\Ga u_\nu^2\,dS_x-\int_\Ga H u_\nu^2\,dS_x,
\end{multline*}
while 
$$
\frac1{R}\,\int_\Ga (u_\nu-R)^2 dS_x=H_0 \int_\Ga u_\nu^2\,dS_x-\int_\Ga u_\nu\,dS_x.
$$
The conclusion then follows by a simple inspection, from these two formulas, \eqref{intermediate}, and 
\eqref{differential-identity}.
\end{proof}

\begin{cor}
Let $\Om$ be a bounded domain with boundary $\Ga$ of class $C^2$ and set $h=q-u$. 
\par
Then it holds that
\begin{multline}
\label{identity-SBT-h}
\frac1{N-1}\int_{\Om} |\na ^2 h|^2 dx+
\frac1{R}\,\int_\Ga (u_\nu-R)^2 dS_x = \\
-\int_{\Ga}(H_0-H)\,h_\nu\,u_\nu\,dS_x+
\int_{\Ga}(H_0-H)\, (u_\nu-R)\,q_\nu\, dS_x.
\end{multline}
\end{cor}

%

\section{Stability for Serrin's overdetermined problem}
\label{sec:stability}


\subsection{Notations} The {\it diameter} of $\Om$ will be indicated by $d_\Om$, while $\de_\Ga (x)$ denotes the distance of a point $x$ to the boundary $\Ga$.
 \par
Even if \eqref{idwps} holds for more general domains, in order to consider the stability issue, we shall assume that $\Om$ is a bounded domain with boundary $\Ga$ of class $C^2$. In fact, under this assumption,
$\Om$ has the properties of the {\it uniform interior and exterior sphere condition}, whose respective radii we have designated by $r_i$ and $r_e$; in other words,
there exists $r_e > 0$ (resp. $r_i>0$) such that for each $p \in \Ga$ there exists a ball contained in $\RR^N \setminus \ol{\Om}$ (resp. contained in $\Om$) of radius $r_e$ (resp. $r_i$) such that
its closure intersects $\Ga$ only at $p$. 
We shall later see that, when $\Om$ is a convex domain, then we can remove the assumption on the regularity of $\Ga$.
\par
The assumed regularity of $\Om$ ensures that the unique solution of \eqref{serrin1} is of class at least $C^{1,\al}(\ol{\Om})$. Thus, we can define
\begin{equation}
\label{bound-gradient}
M=\max_{\ol{\Om}} |\na u|=\max_{\Ga} u_\nu.
\end{equation}
As shown in \cite[Theorem 3.10]{MP}, the following bound holds for $M$:
\begin{equation}
\label{bound-M}
M\le c_N\,\frac{d_\Om(d_\Om+r_e)}{r_e},
\end{equation}
where $c_N=3/2$ for $N=2$ and $c_N=N/2$ for $N\ge 3$. Notice that, when $\Om$ is convex, we can choose $r_e=+\infty$ in \eqref{bound-M} and obtain
\begin{equation}
\label{bound-M-convex}
M\le c_N\,d_\Om.
\end{equation}
For other estimates are present in the literature, see \cite[Remark 3.11]{MP}.

\subsection{Some estimates for harmonic functions}
\label{subsec:harmonicestimates}
As already sketched in the introduction, the desired stability estimate for the spherical symmetry of $\Om$ will be obtained, by means of identity 
\eqref{idwps-h}, if we associate the oscillation of $h$ on $\Ga$
with 
$$
\int_\Om (-u)|\na^2 h|^2 dx.
$$
\par
To fulfill this agenda, we start by relating the factor $(-u)$ appearing in that quantity to the function $\de_\Ga (x)$; we do it in the following lemma.

\begin{lem}
\label{lem:relationdist}
Let $\Om\subset\RR^N$, $N\ge 2$, be a bounded domain such that $\Ga$ is made of regular points for the Dirichlet problem, and let $u$ be the solution of \eqref{serrin1}. Then
$$
-u(x)\ge\frac12\,\de_\Ga(x)^2 \ \mbox{ for every } \ x\in\ol{\Om}.
$$
Moreover, if $\Ga$ is of class $C^{2}$, then it holds that
\begin{equation}
\label{eq:relationdist}
-u(x) \ge \frac{r_i}{2}\,\de_\Ga(x)\ \mbox{ for every } \ x\in\ol{\Om}.
\end{equation}
\end{lem}

\begin{proof}
If every point of $\Ga$ is regular, then a unique solution $u\in C^0(\ol{\Om})\cap C^2(\Om)$ exists for \eqref{serrin1}. Now, for $x\in\Om$, let $r=\de_\Ga(x)$ and consider the ball $B=B_r(x)$. 
If $w$ is the solution of \eqref{serrin1} in $B$, that is $w(y)=(|y-x|^2-r^2)/2$, by comparison we have that $w\ge u$ on $\ol{B}$ and hence, in particular, $w(x) \ge u(x)$, that is we infer the first inequality in the lemma.
\par
If $\Ga$ is of class $C^{2}$, \eqref{eq:relationdist} centainly holds if $\de_\Ga(x)\ge r_i$. If $\de_\Ga (x) < r_i$, instead, let $z$ be the closest point in $\Ga$ to $x$ and call $B$ the ball of radius $r_i$ touching $\Ga$ at $z$ and containing $x$. Up to a translation, we can always suppose that the center of the ball $B$ is the origin $0$. If $w$ is the solution of \eqref{serrin1} in $B$, that is $w(y)=\left(|y|^2- r_i^2 \right)/2$, by comparison we have that $w \ge u$ in $B$, and hence
$$
-u(x) \ge \frac12\,(|x|^2-r_i^2)=
\frac12\,( r_i + |x| )(r_i-|x|)\ge\frac12\,r_i\,(r_i-|x|),
$$
which implies \eqref{eq:relationdist}, since $r_i-|x|=\de_\Ga(x)$.
\end{proof}
\par
By the last lemma, we can estimate the right-hand side of \eqref{idwps-h} from below in terms of the integral 
$$
\int_\Om |\na h|^2 \de_\Ga^{2\al} dx,
$$
with $\al=1$ or $1/2$. For this kind of integral, useful estimates are present in the literature.
We shall briefly report on some of them.

\begin{lem}[Hardy-Poincar\'e-type inequalities]
\label{lem:two-inequalities}
Let $\Om\subset\RR^N$, $N\ge 2$, be a bounded domain with boundary $\Ga$ of class $C^{0,\al}$, and let $x_0$ be a point in $\Om$. Then,
\par
(i) there exists a positive constant $\mu_\al ( \Om)$, such that
\begin{equation}
\label{harmonic-quasi-poincare}
\int_\Om v^2 dx\le \mu_\al ( \Om)^{-1} \int_\Om |\na v|^2 \de_\Ga^{2 \al}  dx,
\end{equation}
for every function $v$ which is harmonic in $\Om$ and such that $v(x_0)=0$;
\par
(ii)
there exists a positive constant, $\ol{\mu}_\al (\Om)$ such that
\begin{equation}
\label{harmonic-poincare}
\int_\Om v^2 dx\le \ol{\mu}_\al (\Om)^{-1} \int_\Om  |\na v|^2 \de_\Ga^{2 \al} dx,
\end{equation}
for every function $v$ which is harmonic in $\Om$ and has mean value zero on $\Om$.
\par
In particular, if $\Ga$ is a Lipschitz boundary, the exponent $2\al$ can be replaced by any exponent in $(0,2]$. 
\end{lem}

\begin{proof}
The assertions (i) and (ii) are easy consequences of a general result of Boas and Straube (see \cite{BS}), that improves a work of Ziemer's (\cite{Zi}).
In case (i), we apply \cite[Example 2.5]{BS}). In case (ii), \cite[Example 2.1]{BS} is appropriate.
\par
The variational problems 
\begin{equation}
\label{inf-harmonic-poincare2}
\mu_\al (\Om) = 
\min \left\{\int_{\Om}  |\na v|^2 \de_\Ga^{2 \al} dx : \int_{\Om} v^2 dx = 1, \,\De v =0 \text{ in } \Om, \, v(x_0) = 0 \right\}
\end{equation}
and
\begin{multline}
\label{inf-harmonic-poincare}
\ol{\mu}_\al (\Om) = 
\min \left\{\int_{\Om}  |\na v|^2 \de_\Ga^{2 \al} dx : \int_{\Om} v^2 \, dx = 1, \,\De v =0 \text{ in } \Om, \int_{\Om}\!v\,dx = 0 \right\}
\end{multline}
then characterize the two constants.
\end{proof}

\begin{rem}{\rm
(i) Notice that in the special case that $\al=0$ from \eqref{harmonic-quasi-poincare} and \eqref{harmonic-poincare} we recover the Poincar\'e-type inequality, that we proved and used in \cite{MP}. In the sequel, we shall use the simplified notation $\ol{\mu}(\Om)=\ol{\mu}_0(\Om)$.
Also, \eqref{harmonic-quasi-poincare} in the case $\al=0$ directly follows from the result in \cite{Zi}.
\par
(ii) We have that
$$
\mu_\al (\Om) \le \ol{\mu}_\al(\Om),
$$
as one can verify by using the function 
$$
v_0=\frac{v-v(x_0)}{1+|\Om|\,v(x_0)^2},
$$
where $v$ is a minimizer for \eqref{inf-harmonic-poincare}.
}
\end{rem}

\medskip

The next lemma, that modifies for our purposes an idea of W. Feldman \cite{Fe}, will be useful to bound the right-hand side of \eqref{idwps-h}.

\begin{lem}[Trace inequality]
\label{lem:improvinglemma}
Let $\Om\subset\RR^N$, $N\ge 2$, be a bounded domain with boundary $\Ga$ of class $C^{2}$
and $z$ be any critical point in $\Om$ of the solution $u$ of \eqref{serrin1}.
\par
The following inequality holds for $h=q-u$, where $q$ is given by \eqref{quadratic}:
\begin{equation}
\label{eq:improvinglemma}
\int_{\Ga} |\na h|^2 dS_x \le \frac{2}{r_i} \left(1+\frac{N}{r_i\, \mu_{1/2} (\Om)} \right)  \int_{\Om} (-u) |\na^2h|^2 dx.
\end{equation}
\end{lem}

\begin{proof}
It is clear that $h\in C^1(\ol{\Om})\cap C^2(\Om)$. We begin with the following differential identity:
\begin{equation*}
\label{diffidimpr}
\dv\,\{v^2 \na u - u \, \na(v^2)\}= v^2 \De u - u \, \De (v^2)= N \, v^2 - 2 u \, |\na v|^2,
\end{equation*}
that holds for any $v$ harmonic function in $\Om$, if $u$ is satisfies\eqref{serrin1}.
Next, we integrate on $\Om$ and, by the divergence theorem, we get:
\begin{equation*}
\label{diffidimprint}
\int_{\Ga} v^2 u_{\nu} \, dS_x = N \int_{\Om} v^2 dx + 2 \int_{\Om} (-u) |\na v|^2 dx.
\end{equation*}
We use this identity with $v=h_i$, and hence we sum up over $i=1,\dots, N$ to obtain:
\begin{equation*}
\int_{\Ga} |\na h|^2 u_{\nu} dS_x = N \int_{\Om} |\na h|^2 dx + 2 \int_{\Om} (-u) |\na^2 h|^2 dx.
\end{equation*}
This formula, together with \eqref{harmonic-quasi-poincare} and \eqref{eq:relationdist}, gives us:
\begin{equation*}
\label{ineqimproving}
\int_{\Ga} |\na h|^2 u_{\nu} dS_x \le 2 \left(1+\frac{N}{r_i\, \mu_{1/2} (\Om)} \right) \int_{\Om} (-u) |\na^2h|^2 dx.
\end{equation*}
The term $u_\nu$ at the left-hand side of this last inequality can be bounded from below by $r_i$, by an adaptation of Hopf's lemma (see also \cite[Theorem 3.10]{MP}). Therefore, \eqref{eq:improvinglemma} follows at once.
\end{proof}

The crucial step in our analysis is Theorem \ref{thm:W22-stability} below, in which we associate the oscillation of $h$,
and hence $\rho_e-\rho_i$, with a weighted $L^2$-norm of its Hessian matrix. 
\par
Before stating that,  we recall from \cite{MP} the following estimate, that links $\rho_e-\rho_i$ with the $L^2$-norm of $h$. 
In fact, if $\Om$ is a bounded domain with boundary of class $C^{2}$, we have that
\begin{equation}
\label{L2-stability}
\rho_e-\rho_i\le a_N\,M^\frac{N}{N+2}\,|\Om|^{-\frac{1}{N}}\, \nr h\nr_2^{ 2/(N+2) },
\end{equation}
for
\begin{equation}
\label{smallness}
\nr h\nr_2\le \al_N \,M\, r_{i}^{\frac{N+2}{2}}.
\end{equation}
The values of the constants $a_N$ and $\al_N$ can be found in \cite[Lemma 3.3]{MP}.

%

\begin{thm}
\label{thm:W22-stability} 
Let $\Om\subset\RR^N$, $N\ge 2$, be a bounded domain with boundary $\Ga$ of class $C^{2}$
and $z$ be any critical point in $\Om$ of the solution $u$ of \eqref{serrin1}.
\par
Consider the function  $h=q-u$, with $q$ given by \eqref{quadratic}, where 
the constant $a$ is chosen such that $h$ has mean value zero on $\Om$.
\par
Then we have that
\begin{equation*}
\rho_e-\rho_i\le C\, \left\{\int_\Om |\na^2 h|^2\de_\Ga\, dx\right\}^{1/(N+2)}
\end{equation*}
if
\begin{equation*}
\int_\Om |\na^2 h|^2 \de_\Ga \, dx<\ve^2.
\end{equation*}
Here,
$$
C=\frac{ a_N\,M^{ \frac{N}{N+2} } }{ |\Om|^{ \frac{1}{N} } \, \left[\ol{\mu}(\Om)\mu_{1/2} (\Om)\right]^{ \frac{1}{N+2} }  }   \quad
\mbox{ and }
\quad
\ve= \al_N \,M \, r_{i}^{\frac{N+2}{2}}\sqrt{\ol{\mu} (\Om)\mu_{1/2} (\Om)},
$$
and the constants $a_N$ and $\al_N$ are those in \eqref{L2-stability} and \eqref{smallness}.
\end{thm}

\begin{proof}
We apply \eqref{harmonic-quasi-poincare} with $x_0=z$ and  $\al=1/2$ to each first derivative of $h$ (since $\na h(z)=0$) and obtain that
\begin{equation*}
\label{serveperlemmafinale}
\int_{\Om} |\na h|^2 \, dx \le \mu_{1/2} (\Om)^{-1} \int_{\Om} | \na^2 h|^2 \de_\Ga \, dx.
\end{equation*}
Hence, we apply \eqref{harmonic-poincare} with $\al=0$ to $h$ and get
\begin{equation*}
\label{eq:serve per lemma prima di main}
\int_{\Om} h^2 \, dx \le \ol{\mu}(\Om)^{-1} \int_\Om |\na h|^2 dx.
\end{equation*}
Thus,
$$
\int_\Om h^2 dx \le \left[\ol{\mu}(\Om)\mu_{1/2} (\Om)\right]^{-1 } \int_\Om |\na^2 h|^2 \de_\Ga\,dx,
$$
and the conclusion follows from \eqref{L2-stability} and \eqref{smallness}.
\end{proof}

\subsection{Stability for Serrin's problem}
\label{subsec:improvedstability}

We collect here our results on the stability of the spherical configuration by putting together the identities of Section \ref{sec:identity} and the estimates in the previous subsection.
\par
Theorem \ref{thm:W22-stability}  above gives an estimate from below of the left-hand side of \eqref{idwps-h}. In this subsection, we will take care of its right-hand side an prove our main result for Serrin's problem.

\begin{thm}[Stability in $L^2$-norm]
\label{thm:Improved-Serrin-stability}
Let $\Om\subset\RR^N$, $N\ge2$, be a bounded domain with boundary $\Ga$ of class $C^2$ and $R$ be the constant defined in \eqref{def-R-H0}.  
\par
Let $u$ be the solution of problem \eqref{serrin1} and 
$z\in\Om$ be any of its critical points.
\par
Then \eqref{balls} holds, with $\rho_i$ and $\rho_e$ given by \eqref{def-rhos} and such that
\begin{equation}
\label{general improved stability serrin}
\rho_e-\rho_i\le C\,\nr u_\nu - R \nr_{2,\Ga}^{2/ (N+2)} \quad \mbox{ if } \quad \nr u_\nu - R \nr_{2,\Ga}<\ve,
\end{equation}
for some positive constants $C$ and $\ve$.
\end{thm}

\begin{proof}
We have that 
$$
\int_\Ga ( R^2-u_\nu^2)\, h_\nu\,dS_x\le (M+R)\,\nr u_\nu-R\nr_{2,\Ga} \nr h_\nu\nr_{2,\Ga},
$$
after an an application of H\"older's inequality.
Thus, by Lemma \ref{lem:improvinglemma}, \eqref{idwps-h}, and this inequality, we infer that
\begin{multline*}
\nr h_\nu\nr_{2,\Ga}^2\le \frac{2}{r_i} \left(1+\frac{N}{r_i\, \mu_{1/2} (\Om)} \right)  \int_{\Om} (-u) |\na^2h|^2 dx \le \\ 
\frac{M+R}{r_i} \left(1+\frac{N}{r_i\, \mu_{1/2} (\Om)} \right)\nr u_\nu-R\nr_{2,\Ga} \nr h_\nu\nr_{2,\Ga},
\end{multline*}
and hence
\begin{equation}
\label{ineq-feldman}
\nr h_\nu\nr_{2,\Ga}\le \frac{M+R}{r_i} \left(1+\frac{N}{r_i\, \mu_{1/2} (\Om)} \right)\nr u_\nu-R\nr_{2,\Ga}.
\end{equation}
Therefore,
\begin{multline*}
\int_\Om |\na ^2 h|^2 \de_\Ga (x)\, dx \le \frac{2}{r_i}\int_\Om (-u) |\na^2 h|^2 dx \le \\
 \left(\frac{M+R}{r_i}\right)^2 \left(1+\frac{N}{r_i\, \mu_{1/2} (\Om)} \right)\nr u_\nu-R\nr_{2,\Ga}^2,
\end{multline*}
by Lemma \ref{lem:relationdist}.
These inequalities and Theorem \ref{thm:W22-stability} then give \eqref{general improved stability serrin}.
\end{proof}

If we want to measure the deviation of $u_\nu$ from $R$ in $L^1$-norm, we get a smaller stability exponent.

\begin{thm}[Stability in $L^1$-norm]
\label{thm:Serrin-stability}
Under the same assumptions of Theorem \ref{thm:Improved-Serrin-stability}, we have that 
\begin{equation}
\label{general stability serrin}
\rho_e-\rho_i\le C\,\nr u_\nu - R \nr_{1,\Ga}^{1/(N+2)} \quad \mbox{ if } \quad \nr u_\nu - R \nr_{1,\Ga}<\ve,
\end{equation}
for some positive constants $C$ and $\ve$.
\end{thm}
\begin{proof}
Instead of applying H\"older's inequality to the right-hand side of \eqref{idwps-h}, we just use the rough bound:
\begin{equation*}
\label{Ineqpassaggio}
\int_{\Om} (-u)\, |\na ^2 h|^2\,dx \le
\frac{1}{2}\, \left( M + R \right)\,(M + d_{\Om}) \, \int_\Ga \left| u_\nu - R \right| \, dS_x,
\end{equation*}
since $(u_\nu+R)\,|h_\nu|\le ( M + R)\,(M + d_{\Om})$ on $\Ga$. The conclusion then follows from similar arguments.
\end{proof}

\begin{rem}[On the constants $C$ and $\ve$]
\label{rem: On the constants}
{\rm
For the sake of clarity, we did not display the values of the constants $C$ and $\ve$ in Theorems, \ref{thm:Improved-Serrin-stability}, \ref{thm:Serrin-stability}, and the relevant theorems in the sequel.
However, their computation will be clear by following the steps of the proofs.
\par
For instance, an inspection of the proofs of  Theorem \ref{thm:Improved-Serrin-stability} informs us that the constants in \eqref{general improved stability serrin} are:
\begin{equation*}
\label{improved-def-C}
C=a_N \, \frac{ M^{ \frac{N}{N+2} } \, \left( M + R \right)^\frac{2}{N+2}}{ |\Om|^{ \frac{1}{N} } \, r_i^\frac{3}{N+2}}  
\left\{\frac{N+r_i \, \mu_{1/2} (\Om)}{\ol{\mu}(\Om) \, \mu_{1/2} (\Om)^2}\right\}^\frac{1}{N+2}
\end{equation*}
and 
\begin{equation*}
\label{improved-def-eps}
\ve=\al_N \,\frac{M}{M+R} \, \frac{ \ol{\mu}(\Om)^{\frac{1}{2} } \,\mu_{1/2} (\Om)}{ \sqrt{ N + r_i \, \mu_{1/2} (\Om)} }  \, r_{i}^{\frac{N+5}{2}}.
\end{equation*}
Here, the constants $a_N$ and $\al_N$ are those appearing in \eqref{L2-stability} and \eqref{smallness}.
\par
The constants $C$ and $\ve$ can be shown to depend only on some geometric parameters of $\Om$. In fact, we can use \eqref{bound-M} to bound $M$ in terms of $d_\Om$ and $r_e$. To estimate the ratio $R$,  we can take advantage of the {\it isoperimetric} inequality, to bound $|\Ga|$ from below in terms of $|\Om|^{1-1/N}$, and then the {\it isodiametric} inequality, to bound $|\Om|^{1/N}$ from above in terms of $d_\Om$. 
\par
In the next remark, we show how to deal with $\ol{\mu}(\Om)$ and $\mu_{1/2} (\Om)$.
}
\end{rem}

\begin{rem}[Estimating $\ol{\mu}(\Om)$ and $\mu_{1/2}(\Om)$]
\label{rem:John estimates}
{\rm
(i) For simplicity, in what follows, $k_N$ denotes a positive number, that only depends on $N$, whose value may change from line to line.
\par
A lower bound of $\ol{\mu}(\Om)$ can be obtained as follows.
In \cite[Theorem 1.3]{HS} a more general form of inequality \eqref{harmonic-poincare} (without the assumption of harmonicity) is proved for the class of the $b_0$-John domains (see \cite{HS} for the definition) and the constants are explicitly computed; in particular, with the aid of  \cite{HS}, we can easily deduce that
$$
\ol{\mu}(\Om)^{-1} \le k_N \, |\Om|^{\frac{2}{N}} \, b_0^{2 N}, 
$$
for some constant $k_N$ only depending on $N$.
A domain of class $C^{2}$ is obviously a $b_0$-John domain and it is not difficult to show that 
\begin{equation*}
\label{John-parameter-estimate}
b_0 \le  d_\Om /  r_i ,
\end{equation*}
and hence obtain that
\begin{equation}
\label{bound-ol-mu}
\ol{\mu}(\Om)^{-1} \le k_N \, |\Om|^{\frac{2}{N}} \left(\frac{d_\Om}{r_i}\right)^{2 N}.
\end{equation}
An alternative way to estimate $\ol{\mu}(\Om)$ can be found in \cite[Remark 3.8 (ii)]{MP}.
\par
A lower bound for $\mu_{1/2}(\Om)$ can also be obtained, but it may depend on choice of the particular critical point of $u$. In fact,
by following the argument of \cite[Lemma 8]{Fe}, one can adapt the result contained in \cite[Theorem 1.3]{HS} to the case of harmonic functions vanishing at a given point $x_0$ (i.e the case of \eqref{harmonic-quasi-poincare}). To explicitly perform the computations, the definition of 
bounded $L_0$-John domain with base point $x_0$, given in \cite{Fe}, is appropriate.
In fact, for that class of domains, by means of \cite[Lemma 8]{Fe}, we can deduce that 
\begin{equation*}
\label{Johnineqest}
\mu_{1/2}(\Om)^{-1} \le k_N \, | \Om|^\frac{1}{N} \, L_0^{2N}
\end{equation*}
and, by the definition, it is not difficult to prove the following bound:
\begin{equation*}
\label{LJohnparameterestimate}
L_0 \le \frac{d_\Om}{\min[r_i, \de_\Ga (x_0)] } .
\end{equation*}
Thus, we obtain:
\begin{equation}
\label{bound-mu-1/2}
\mu_{1/2}(\Om)^{-1} \le k_N \, | \Om|^\frac{1}{N} \left(\frac{d_\Om}{\min[r_i, \de_\Ga (x_0)]}\right)^{2N}.
\end{equation}
\par
Finally, it is clear that we can eliminate the dependence on $|\Om|$ in \eqref{bound-ol-mu} and \eqref{bound-mu-1/2} in favour of $d_\Om$ again by using the isodiametric inequality. 

\medskip

(ii)
When $\Om$ is convex, we can further eliminate the dependence on $\de_\Ga (x_0)$ appearing in \eqref{bound-mu-1/2}. In this case,
$u$ has a unique minimum point $z$ in $\Om$, since $u$ is analytic and the level sets of $u$ are convex by a result in \cite{Ko} (see \cite{MS}, for a similar argument).  Thereby, we have only one choice for the point $x_0$, that is $x_0=z$. 
\par
Thus, $\de_\Om(z)$ may be estimated from below by following arguments established in \cite{BMS}).
First, by putting toghether the arguments in \cite[Theorem 2.7]{BMS} and \cite[Remark 2.5]{BMS}, we get that
$$
\de_\Ga(z)\ge \frac{k_N}{|\Om|\, d_\Om^{N-1}}\,\max_{\ol{\Om}}(-u)
$$
and, by a simple comparison argument, the maximum can be bounded from below by 
$r_i^2/2.$
Thus, we have that
$$
\de_\Ga(z)\ge k_N\,\frac{r_i^{2N}}{|\Om|\, d_\Om^{N-1}}
$$
and, again, we can replace $|\Om|$ by $d_\Om$, owing to the isodiametric inequality. 
Moreover, in this case, the dependence on $r_e$ for $M$ can be removed, thanks to \eqref{bound-M-convex}. 
\par
By similar arguments, we can take care of $\ve$. All in all, we can affirm that $C$ and $\ve$ are controlled by $r_i$ and $d_\Om$, only. 
}
\end{rem}

\bigskip

Since the estimates in Theorems \ref{thm:Improved-Serrin-stability} and \ref{thm:Serrin-stability} do not depend on the particular critical point chosen, as a corollary, we obtain results of closeness to a union of balls, similar to \cite[Corollary 4.3]{MP}: here, we just illustrate the instance of Theorem \ref{thm:Improved-Serrin-stability}.

\begin{cor}[Closeness to an aggregate of balls]
\label{cor:SBT-stability}
Let $\Ga$, $R$ and $u$ be as in Theorem \ref{thm:Improved-Serrin-stability}.
\par
Then, there exist points $z_1, \dots, z_n$ in $\Om$, $n\ge 1$, and corresponding numbers
\begin{equation*}
\label{def-rho_j}
\rho_i^j=\min_{x\in\Ga}|x-z_j| \ \mbox{ and } \ \rho_e^j=\min_{x\in\Ga}|x-z_j|,
\quad j=1, \dots, n,
\end{equation*}
such that
\begin{equation*}
\label{aggregate}
\bigcup_{j=1}^n B_{\rho_i^j}(z_j)\subset\Om\subset \bigcap_{j=1}^n B_{\rho_e^j}(z_j)
\end{equation*}
and
$$
\max_{1\le j\le n}(\rho_e^j-\rho_i^j)\le C\,\nr u_\nu - R \nr_{2,\Ga}^{2/(N+2)} 
\quad \mbox{ if } \quad \nr u_\nu - R \nr_{2,\Ga}<\ve,
$$
for some positive constants $C$ and $\ve$.
\end{cor}

\begin{proof}
We pick one point $z_j$ from
each connected component of the set of local minimum points of $u$. By applying Theorem \ref{thm:Improved-Serrin-stability} to each $z_j$, the conclusion is then evident.
\end{proof}

\section{Enhanced stability for Alexandrov's Soap Bubble Theorem}
\label{sec:SBT-improvedstab}

In this section, we collect those benefits of the new estimate \eqref{general improved stability serrin} that affect the stability issue for the Soap Bubble Theorem. 
\par
We begin by recalling a couple of inequalities concerning the harmonic function $h$, that we obtained in \cite[Theorem 3.4]{MP} and, in the sequel, will play the role of those of Theorem \ref{thm:W22-stability}:
\begin{equation}
\label{Lipschitz-stability}
\rho_e-\rho_i\le C_0\,\left(\int_\Om |\na^2 h|^2 dx\right)^{1/2},
\end{equation}
if $N=2, 3$, and 
\begin{equation}
\label{W22-stability}
\rho_e-\rho_i\le C_0\, \left(\int_\Om |\na^2 h|^2 dx\right)^{1/(N+2)} \ \mbox{ for } \ \int_\Om |\na^2 h|^2 dx<\ve_0^2,
\end{equation}
that holds for $N\ge 4$. The point $z$ chosen in the definition \eqref{def-rhos} of $\rho_i$ and $\rho_e$ is any critical point of $u$, as usual. The positive constants $C_0$ and $\ve_0$ depend on $N$, $|\Om|$, $d_\Om$, $r_i$ $r_e$, and $\mu_0(\Om)$. For their values, we refer to \cite[Theorem 3.4]{MP}.

\begin{rem}
\label{rem:estimate of mu_0}
{\rm
The parameter $\mu_0 (\Om)$ can be estimated by following the arguments used in item (i) of Remark \ref{rem:John estimates} to estimate $\mu_{1/2} (\Om)$. In fact, we deduce that 
$$
\mu_0 ( \Om )^{-1} \le k_N \, |\Om|^\frac{2}{N} \, L_0^{2 N}.
$$
Another way to estimate $\mu_0 (\Om)$ can be found in \cite[Remark 3.8 (iii)]{MP}.
}
\end{rem}

\bigskip

Next, we derive the following lemma, that parallels and is a useful consequence of Lemma \ref{lem:improvinglemma}.

\begin{lem}
\label{lem:improving-SBT}
Let $N\ge 2$ and let $\Ga$ be the connected boundary of class $C^{2}$ of a bounded domain $\Om\subset\RR^N$. Denote by $H$ the mean curvature function for $\Ga$ and let $H_0$ be the constant defined in \eqref{def-R-H0}.
\par
Then, the following inequality holds:
\begin{equation}
\label{improving-SBT}
\nr  u_\nu-R\nr_{2,\Ga} \le \\ 
R\left\{ d_\Om +\frac{M (M+R)}{r_i}\left(1+\frac{N}{r_i\,\mu_{1/2}(\Om)}\right)\right\} \nr H_0-H\nr_{2,\Ga}.
\end{equation}
\end{lem}

\begin{proof}
Discarding the first summand on the left-hand side of \eqref{identity-SBT-h} and applying H\"older's inequality on its right-hand side gives that
$$
\frac1{R} \nr  u_\nu-R\nr_{2,\Ga}^2 \le \nr H_0-H\nr_{2,\Ga}\left( M \nr h_\nu\nr_{2,\Ga} + d_\Om \,\nr  u_\nu-R\nr_{2,\Ga} \right),
$$ 
since $u_\nu\le M$ and $|q_\nu|\le d_\Om $ on $\Ga$. Thus, inequality \eqref{ineq-feldman}
implies that
\begin{multline*}
\nr  u_\nu-R\nr_{2,\Ga}^2 \le \\ 
R\left\{ d_\Om +\frac{M (M+R)}{r_i}\left(1+\frac{N}{r_i\,\mu_{1/2}(\Om)}\right)\right\} \nr H_0-H\nr_{2,\Ga} \nr  u_\nu-R\nr_{2,\Ga},
\end{multline*}
from which \eqref{improving-SBT} follows at once.
\end{proof}

\begin{thm}[Stability for the Soap Bubble Theorem] 
\label{thm:SBT-improved-stability}
Let $N\ge 2$ and let $\Ga$ be the connected boundary of class $C^{2}$ of a bounded domain $\Om\subset\RR^N$. Denote by $H$ the mean curvature function for $\Ga$ and let $H_0$ be the constant defined in \eqref{def-R-H0}.
\par
There is a point $z\in\Om$ such that
\begin{enumerate}[(i)]
\item
if $N=2$ or $N=3$, then
\begin{equation}
\label{SBT-improved-stability-Lipschitz}
\rho_e-\rho_i\le C\,\nr H_0-H\nr_{2,\Ga},
\end{equation}
for some positive constant $C$;

\item
If $N\ge 4$, then
\begin{equation}
\label{SBT-improved-stability}
\rho_e-\rho_i\le C\,\nr H_0-H\nr_{2,\Ga}^{2/(N+2)} \quad \mbox{ if } \quad \nr H_0-H\nr_{2,\Ga}<\ve,
\end{equation}
for some positive constants $C$ and $\ve$.
\end{enumerate}
\end{thm}

\begin{proof}
Discarding the second summand on the left-hand side of \eqref{identity-SBT-h} and applying H\"older's inequality on its right-hand side, as in the previous proof, gives that
\begin{multline*}
\frac1{N-1}\,\int_\Om |\na^2 h|^2 dx\le \\ 
R\left\{ d_\Om +\frac{M (M+R)}{r_i}\left(1+\frac{N}{r_i\,\mu_{1/2}(\Om)}\right)\right\} \nr H_0-H\nr_{2,\Ga} \nr  u_\nu-R\nr_{2,\Ga}, \le \\
R^2\left\{ d_\Om +\frac{M (M+R)}{r_i}\left(1+\frac{N}{r_i\,\mu_{1/2}(\Om)}\right)\right\}^2 \nr H_0-H\nr_{2,\Ga}^2,
\end{multline*}
where the second inequality follows from Lemma \ref{lem:improving-SBT}.
\par
Inequalities \eqref{SBT-improved-stability-Lipschitz} and \eqref{SBT-improved-stability} then result from \eqref{Lipschitz-stability} and \eqref{W22-stability}.
\end{proof}

\begin{rem}[On the constants $C$ and $\ve$]{\rm
Needless to say, the proof of Theorem \ref{thm:SBT-improved-stability} tells us that, by proceeding as in Remarks \ref{rem: On the constants}, \ref{rem:John estimates}, and \ref{rem:estimate of mu_0}, we can reduce the dependence of $C$ and $\ve$ to the parameters $N$, $r_i$, $r_e$, $d_\Om$, and $\de_\Ga(z)$, if $\Ga$ is of class $C^2$, and to $N$, $r_i$, and $d_\Om$, if $\Om$ is also convex.
}
\end{rem}

\begin{rem}
\label{rem:proofSBTwithSerrin}{\rm

Assertion (ii) of Theorem \ref{thm:SBT-improved-stability} can also be proved as a direct corollary of Theorem \ref{thm:Improved-Serrin-stability}, by noting that relation \eqref{improving-SBT} together with \eqref{general improved stability serrin} gives \eqref{SBT-improved-stability}. 
}
\end{rem}

\begin{rem}{\rm
Analogs of Theorem \ref{thm:Serrin-stability} and Corollary \ref{cor:SBT-stability} can be easily derived by following the steps of their proofs.
}\end{rem}

\begin{rem}
{\rm
The assumption of smallness of the relevant deviation $\eta$ required in Theorems \ref{thm:Improved-Serrin-stability}, \ref{thm:Serrin-stability}, Corollary \ref{cor:SBT-stability} and in (ii) of Theorem \ref{thm:SBT-improved-stability} is only apparent, because if $\eta\ge \ve$, then it is a trivial matter to obtain an upper bound for $\rho_e - \rho_i$ in terms of $\eta$. Thus, all the stability estimates that we presented are {\it global}.
}
\end{rem}
 
Another consequence of Lemma \ref{lem:improving-SBT} is the following inequality that shows an {\it optimal} stability exponent for any $N\ge 2$. The number $\cA(\Om)$, defined in \eqref{asymmetry}, is some sort of asymmetry similar to the so-called {\it Fraenkel asymmetry} (see \cite{Fr}).

\begin{thm}[Stability by asymmetry]
\label{th:asymmetry}
Let $N\ge 2$ and let $\Ga$ be the connected boundary of class $C^{2}$ of a bounded domain $\Om\subset\RR^N$. Denote by $H$ the mean curvature function for $\Ga$ and let $H_0$ be the constant defined in \eqref{def-R-H0}.
\par
Then it holds that
\begin{equation}
\label{eq:SBTstabinmisura}
\cA(\Om)\le C \, \nr H_0 - H \nr_{2,\Ga},
\end{equation}
for some positive constant $C$.
\end{thm}

\begin{proof}
We use \cite[inequality (2.14)]{Fe}: we have that 
\begin{equation*}
\label{Feldman's estimate}
\frac{|\Om \De B_R^z|}{|B_R^z|} \le C \, \nr u_\nu - R \nr_{2,\Ga},
\end{equation*}
where $B_R^z$ is a ball of radius $R$ (as defined in \eqref{def-R-H0}) and centered at a minimum point $z$ of $u$. Hence, we obtain that
$$
\cA(\Om)\le C \, \nr u_\nu - R \nr_{2,\Ga},
$$
by the definition \eqref{asymmetry}.
Thus, thanks to \eqref{improving-SBT}, we obtain \eqref{eq:SBTstabinmisura}.
\par
Here, if we estimate $L_0$ as done in Remark 3.9, we can see that $C$ depends on $N, d_\Om, r_i, |\Ga|/|\Om|$, and $\de_{\Ga} (z)$ if $\Ga$ is of class $C^2$; the dependence on the last parameter can be removed when $\Om$ is also convex.
\end{proof}

\begin{rem}[On the asymmetry $\cA(\Om)$]{\rm

Notice that, for any $x\in \Om$, we have that 

$$
\frac{|\Om \De B_R^x|}{|B_R^x|} \le \frac{|B_{\rho_e}^x \setminus B_{\rho_i}^x|}{|B_R^x|}=
\frac{\rho_e^N-\rho_i^N}{R^N} \le \frac{N \, \rho_e^{N-1}}{R^N} \, (\rho_e - \rho_i),
$$
and $\rho_e\le d_\Om$.
Thus, if $d_\Om/R$ remains bounded and $(\rho_e - \rho_i)/R$ tends to $0$, then the ratio $|\Om \De B_R^x|/|B_R^x|$ does it too.
\par
The converse is not true in general. For example, consider a lollipop made by a ball and a stick with fixed length $L$ and vanishing width; as that width vanishes, the ratio $d_\Om/R$ remains bounded, while $|\Om \De B_R^x|/|B_R^x|$ tends to zero and $\rho_e - \rho_i \ge L >0$.

\bigskip

If we fix $r_i$ and $r_e$, we have the following result.
}
\end{rem}

\begin{thm}
\label{thm:comparing asymmetry}
Let $\Om\subset\RR^N$, $N\ge 2$, be a bounded domain satisfying the uniform interior and exterior sphere conditions with radii $r_i$ and $r_e$.
\par
Then, we have that
$$
\rho_e - \rho_i \le 4 \, R \, \cA(\Om)^{\frac{1}{N}} \quad \mbox{ if } \quad \cA(\Om)\le \left( \frac{r_i}{R} \right)^N .
$$
\end{thm}
\begin{proof}
Let $x$ be any point in $\Om$.
It is clear that 
\begin{equation}
\label{maxcasiincomparing asymmetry}
\max(\rho_e - R, R- \rho_i) \ge \frac{\rho_e - \rho_i}{2}.
\end{equation}
If that maximum is $\rho_e - R$, at a point $y$ where the ball centered at $x$ with radius $\rho_e$ touches $\Ga$, we consider the interior touching ball $B_{r_i}$ . 
\par
If $2 r_i < (\rho_e - \rho_i)/2$ then $B_{r_i} \subset \Om \setminus B_R^x$ and hence \begin{equation*}
\frac{|\Om \De B_R^x|}{|B_R^x|} \ge \left( \frac{r_i}{R} \right)^N.
\end{equation*}
If, else,  $2 r_i \ge (\rho_e - \rho_i)/2$, $B_{r_i}$ contains a ball of radius $(\rho_e - \rho_i)/4$ still touching $\Ga$ at $y$. Such a ball is contained in $\Om \setminus B_R^x$, and hence
\begin{equation*}
\frac{|\Om \De B_R^x|}{|B_R^x|} \ge \left( \frac{\rho_e - \rho_i}{4 \, R} \right)^N.
\end{equation*}
\par
Thus, we proved that
\begin{equation*}
\frac{|\Om \De B_R^x|}{|B_R^x|} \ge \min \left\lbrace \left( \frac{r_i}{R} \right)^N , \left( \frac{\rho_e - \rho_i}{4 \, R} \right)^N \right\rbrace,
\end{equation*}
and the conclusion easily follows, since $x$ was arbitrarily chosen in $\Om$. 
\par
If, else, the maximum in \eqref{maxcasiincomparing asymmetry} is $R- \rho_i$, we proceed similarly, by reasoning on the exterior ball $B_{r_e}$ and $\RR^N\setminus\ol{\Om}$, instead.
\end{proof}

\begin{rem}
{\rm
Theorem \ref{thm:comparing asymmetry} and \eqref{eq:SBTstabinmisura} give the inequality
$$
\rho_e - \rho_i \le C \, \nr H_0 - H \nr_{2,\Ga}^{1/N}
$$
that, for any $N\ge 2$, is poorer than what obtained in \ref{thm:SBT-improved-stability}.
}
\end{rem}

\section*{Acknowledgements}
The paper was partially supported by the Gruppo Nazionale Analisi Matematica Probabilit\`a e Applicazioni (GNAMPA) of the Istituto Nazionale di Alta Matematica (INdAM).

\end{document}